\documentclass[12pt]{amsart}
\usepackage{amsfonts,amssymb,amscd,amsmath,enumerate,verbatim,calc}
\usepackage{graphicx}

\newtheorem{Theorem}{Theorem}[section]
\newtheorem{Lemma}[Theorem]{Lemma}

\newtheorem{Proposition}[Theorem]{Proposition}
\newtheorem{Remark}[Theorem]{Remark}

\begin{document}
\title{On the Distinguishing number of Functigraphs }
\author{Muhammad Fazil, Muhammad Murtaza, Usman Ali, Imran Javaid}
\keywords{distinguishing number, functigraph.
$^*$ Corresponding author: mfazil@bzu.edu.pk}

\address{Centre for advanced studies in Pure and Applied Mathematics,
Bahauddin Zakariya University Multan, Pakistan
\newline Email: mfazil@bzu.edu.pk, mahru830@gmail.com,
uali@bzu.edu.pk, imran.javaid@bzu.edu.pk}

\date{}
\maketitle
\begin{abstract} Let $G_{1}$ and $G_{2}$ be disjoint copies of
a graph $G$, and let $g:V(G_{1})\rightarrow V(G_{2})$ be a function.
A functigraph $F_{G}$ consists of the vertex set $V(G_{1})\cup
V(G_{2})$ and the edge set $E(G_{1})\cup E(G_{2})\cup
\{uv:g(u)=v\}$. In this paper, we extend the study of the
distinguishing number of a graph to its functigraph. We discuss the
behavior of the distinguishing number in passing from $G$ to $F_{G}$
and find its sharp lower and upper bounds. We also discuss the
distinguishing number of functigraphs of complete graphs and join
graphs.
\end{abstract}

\section{Preliminaries}
Given a key ring of apparently identical keys to open different
doors, how many colors are needed to identify them? This puzzle was
given by Rubin \cite{frank rubin} in 1980 for the first time. In
this puzzle, there is no need for coloring to be proper. Indeed, one
cannot find a reason why adjacent keys must be assigned different
colors, whereas in other problems like storing chemicals, scheduling
meetings a proper coloring is needed, and one with a small number of
colors is required.


From the inspiration of this puzzle, Albertson and Collins
\cite{alb} introduced the concept of the distinguishing number of a
graph as follows: A labeling $f: V(G)\rightarrow \{1,2,3,...,t\}$ is
called a $t$-\emph{distinguishing} if no non-trivial automorphism of
a graph $G$ preserves the vertex labels. The \emph{distinguishing
number} of a graph $G$, denoted by $Dist(G)$, is the least integer
$t$ such that $G$ has $t$-distinguishing labeling. For example, the
distinguishing number of a complete graph $K_n$ is $n$, the
distinguishing number of a path $P_n$ is $2$ and the distinguishing
number of a cycle $C_{n},\ n\geq 6$ is $2$. For a graph $G$ of order
$n$, $1\le Dist(G)\le n$ \cite{alb}. If $H$ is a subgraph of a graph
$G$ such that automorphism group of $H$ is a subset of automorphism
group of $G$, then $Dist(H)\le Dist(G)$.


Harary \cite{har} gave different methods (orienting some of the
edges, coloring some of the vertices with one or more colors and
same for the edges, labeling vertices or edges, adding or deleting
vertices or edges) of destroying the symmetries of a graph. Collins
and Trenk defined the distinguishing chromatic number in \cite{col}
where they used proper $t$-distinguishing for vertex labeling. They
have also given a comparison between the distinguishing number, the
distinguishing chromatic number and the chromatic number for
families like complete graphs, paths, cycles, Petersen graph and
trees etc. Kalinowski and Pilsniak \cite{Pil} have defined similar
graph parameters, the distinguishing index and the distinguishing
chromatic index, they labeled edges instead of vertices. They have
also given a comparison between the distinguishing number and the
distinguishing index for a connected graph $G$ of order $n\ge 3$.
Boutin \cite{bou} introduced the concept of determining sets. In
\cite{albrt+bou+2}, Albertson and Boutin proved that a graph is
$t$-distinguishable if and only if it has a determining set that is
$(t-1)$-distinguishable. They also proved that every Kneser graph
$K_{n:k}$ with $n\ge 6$ and $k\ge 2$ is 2-distinguishable. A
considerable literature has been developed in this area see
\cite{albertson,albrt+bou+1,
bog,mch1,mch2,cheng,kla1,kla2,tymoczko}.

Unless otherwise specified, all the graphs $G$ considered in this
paper are simple, non-trivial and connected. 
The \emph{open neighborhood} of a vertex $u$ of $G$ is $N(u)=\{v\in
V(G):uv\in E(G)\}$ and the \emph{closed neighborhood} of $u$ is
$N(u)\cup \{u\}$. Two vertices $u,v$ are \emph{adjacent twins} if
$N[u]=N[v]$ and \emph{non adjacent twins} if $N(u)=N(v)$. If $u,v$
are adjacent or non adjacent twins, then $u,v$ are \emph{twins}. A
set of vertices is called \emph{twin-set} if every of its two
vertices are twins. A graph $H$ is said to be a \emph{subgraph} of a
graph $G$ if $V(H) \subseteq V(G)$ and $E(H) \subseteq E(G)$. Let
$S\subset V(G)$ be any subset of vertices of $G$. The \emph{induced
subgraph}, denoted by $<S>$, is the graph whose vertex set is $S$
and whose edge set is the set of all those edges in $E(G)$ which
have both end vertices in $S$.


The idea of permutation graph was introduced by Chartrand and Harary
\cite{char} for the first time. They defined the permutation graph
as follows: a permutation graph consists of two identical disjoint
copies of a graph $G$, say $G_1$ and $G_2$, along with $|V(G)|$
additional edges joining $V(G_1)$ and $V(G_2)$ according to a given
permutation on $\{1, 2, . . . , |V(G)|\}$. Dorfler \cite{dor},
introduced a mapping graph which consists of two disjoint identical
copies of graph where the edges between the two vertex sets are
specified by a function. The mapping graph was rediscovered and
studied by Chen et al. \cite{yi}, where it was called the
functigraph. A functigraph is an extension of permutation graph.
Formally the functigraph is defined as follows: Let $G_{1}$ and
$G_{2}$ be disjoint copies of a connected graph $G$, and let
$g:V(G_{1})\rightarrow V(G_{2})$ be a function. A \emph{functigraph}
$F_{G}$ of a graph $G$ consists of the vertex set $V(G_{1})\cup
V(G_{2})$ and the edge set $E(G_{1})\cup E(G_{2})\cup
\{uv:g(u)=v\}$. Linda et al. \cite {kang,kang1} and Kang et al.
\cite {kang2} have studied the functigraph for some graph invariants
like metric dimension, domination and zero forcing number. In
\cite{fazil2}, we have studied the fixing number of functigraph. The
aim of this paper is to study the distinguishing number of
functigraph.

Throughout the paper, we will denote the set of all automorphisms of
a graph $G$ by $\Gamma(G)$, the functigraph of $G$ by $F_{G}$,
$V(G_{1})=A$, $V(G_{2})=B$, $g:A\rightarrow B$ is a function,
$g(V(G_{1}))= I$, $|g(V(G_{1}))|= |I|= s$.

This paper is organized as follows. In Section 2, we give sharp
lower and upper bounds for distinguishing number of functigraph.
This section also establishes the connections between the
distinguishing number of graphs and their corresponding functigraphs
in the form of realizable results. In Section 3, we provide the
distinguishing number of functigraphs of complete graphs and join of
path graphs. Some useful results related to these families have also
been presented in this section.

\section{Bounds and some realizable results}

 The sharp lower and upper bounds on the distinguishing number of
functigraphs are given in the following result.
\begin{Proposition} \label{f5}Let $G$ be a connected graph of order $n\geq 2$, then $$1\leq
Dist(F_{G})\leq Dist(G)+1.$$ Both bounds are sharp.
\end{Proposition}
\begin{proof} Obviously, $1\leq Dist(F_{G})$ by definition. Let $Dist(G)=t$
 and $f$ be a $t$-distinguishing labeling for graph $G$. Also, let $u_i\in A$ and
 $v_i\in B$, $1\le i \le n$. We extend labeling
 $f$ to $F_G$ as: $f(u_i)=f(v_i)$ for all $1\le i \le n$. We have following two cases for
 $g$:
 \begin{enumerate}
    \item If $g$ is not bijective, then $f$ as defined earlier is a
    $t$-distinguishing labeling for $F_G$. Hence, $Dist(F_G)\le t$.
    \item If $g$ is bijective, then $f$ as defined earlier
    destroys all non-trivial automorphisms of $F_G$ except the flipping
     of $G_1$ and $G_2$ in $F_G$, for some choices
    of $g$. Thus, $Dist(F_G)\le t+1$.
    \end{enumerate}
For the sharpness of the lower bound, take $G=P_{3}$ and
$g:A\rightarrow B$, be a function such that $g(u_{i})= v_{1}, i=1,2$
and $g(u_{3})= v_{3}$. For the sharpness of the upper bound, take
$G$ as rigid graph and $g$ as identity function.
\end{proof}
Since at least $m$ colors are required to break all automorphisms of
a twin set of cardinality $m$, so we have the following corollary.
\begin{Proposition}\label{rem3.3111}
Let $U_1,U_2,...,U_t$ be disjoint twin sets in a connected graph $G$
of order $n\geq 3$ and $m= max\{|U_i|: 1\leq i\leq t\}$,\\ (i)
$Dist(G)\ge m$.
\\ (ii) If $Dist(G)=m$, then $Dist(F_G)\le m$.
\end{Proposition}

\begin{Lemma}\label{Lemma g constant}
Let $G$ be a connected graph of order $n\ge 2$ and $g$ be a constant
function, then $Dist(F_G)=Dist(G)$.
\end{Lemma}
\begin{proof}
Let $I=\{v\}\subset B$. Then $\Gamma(G)=\Gamma(< A\cup
\{v\}>)\subset \Gamma(F_G)$. Thus, vertices in $A\cup \{v\}$ are
labeled by $Dist(G)$ colors. Since $g$ is a constant function,
therefore all vertices in $V(F_G)\setminus \{A\cup\{v\}\}$ are not
similar to any vertex in $A\cup\{v\}$ in functigraph $F_G$.
Therefore, vertices in $V(F_G)\setminus\{A\cup\{v\}\}$ can also be
labeled from these $Dist(G)$ colors. Hence, $Dist(F_G)=Dist(G)$.

\end{proof}

\begin{Remark}\label{rem3.31115}
Let $G$ be a connected graph and $Dist(F_G)=m_1$ if $g$ is constant
and $Dist(F_G)=m_2$ if $g$ is not constant, then $m_1\geq m_2.$
\end{Remark}

A vertex $v$ of degree at least three in a connected graph $G$ is
called a \emph{major vertex}. Two paths rooted from the same major
vertex and having the same length are called the \emph{twin stems}.

We define a function $\psi: \mathbb{N} \setminus\{1\}\rightarrow
\mathbb{N}\setminus \{1\}$ as $\psi (m)=k$ where $k$ is the least
number such that $m\le 2{k\choose2}+k$. For example, $\psi(19)=5$.
Note that $\psi$ is well-defined.

\begin{Lemma}\label{LemmaTwinStem}
If a graph $G$ has $t\ge 2$ twin stems of length 2 rooted at same
major vertex, then $Dist(G)\ge \psi(t)$.
\end{Lemma}

\begin{proof}
Let $x\in V(G)$ be a major vertex and $xu_iu'_i$ where $1\leq i\leq
t$ are twin stems of length 2 attach with $x$. Let
$H=<\{x,u_i,u'_i\}>$ and $k=\psi(t)$. We define a labeling $f:V(H)
\rightarrow \{1,2,...,k\}$ as: $$f(x)=k,$$
\begin{equation}\label{array1}
f({u_i})=\left\{
\begin{array}{ll}
1\,\,\,\,\,\,\,\,\,\,\,\,\,\,\,\,if \,\,\,\,\,\,\,\,\,\,1\leq i\,\leq k\\
2\,\,\,\,\,\,\,\,\,\,\,\,\,\,\,\,if \,\,\,\,\,\,\,\,\,\,k+1\leq i\,\leq 2k\\
3\,\,\,\,\,\,\,\,\,\,\,\,\,\,\,\,if \,\,\,\,\,\,\,\,\,\,2k+1\leq i\,\leq 3k\\
\vdots \,\,\,\,\,\,\,\,\,\,\,\,\,\,\,\,\,\,\, \,\,\,\,\,\,\,\,\,\,\,\,\,\,\,\,\,\,\,\vdots \\
k\,\,\,\,\,\,\,\,\,\,\,\,\,\,\,\,if \,\,\,\,\,\,\,\,\,\,(k-1)k+1\leq i\,\leq k^2\\
\end{array}
\right.
\end{equation}

\begin{equation}\label{array2}
f({u'_i})=\left\{
\begin{array}{ll}
i\, \mbox{mod} (k)\,\,\,\,\,\,\,if \,\,\,\,\,\,\,\,\,\, 1\le i\,\mbox{mod} (k) \le k-1, &  \\
k\,\,\,\,\,\,\,\,\,\,\,\,\,\,\,\,\,\,\,\,\,\,\,\,\,\,if \,\,\,\,\,\,\,\,\,\,\,\,\,\,\,\,\,\,\,\,\,  i\,\mbox{mod}(k) =0,&  \\

\end{array}
\right.
\end{equation}

Using this labeling, one can see that $f$ is a $t$-distinguishing
for $H$. Since permutations with repetition of $k$ colors, when 2 of
them are taken at a time is equal to $2{k\choose 2}+k$, therefore at
least $k$ colors are needed to label the vertices in $t$-stems.
Hence, $k$ is the least integer for which $G$ has $k$-distinguishing
labeling. Since $\Gamma(H)\subseteq \Gamma(G)$, therefore
$Dist(G)\ge \psi(t)$.
\end{proof}

\begin{figure}[h]
        \centerline
        {\includegraphics[width=10cm]{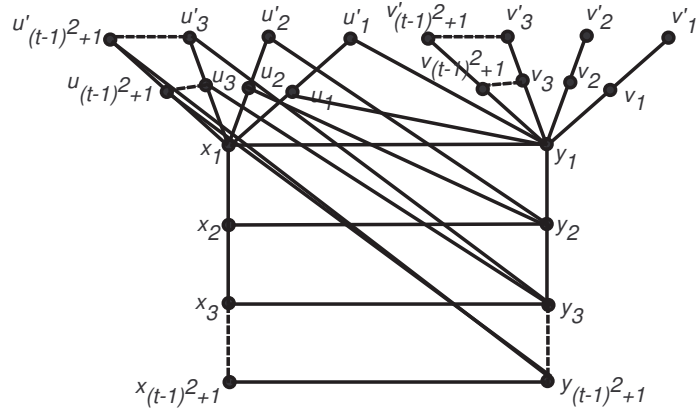}}
        \label{Fix3}
        \caption{Graph with $Dist(G)= t = Dist(F_{G}).$}\label{f1}
\end{figure}

\begin{Lemma}\label{fff} For any integer $t\geq 2$,
there exists a connected graph $G$ and a function $g$ such that
$Dist(G)=t=Dist(F_{G})$.
\end{Lemma}
\begin{proof} Construct the graph $G$ as follows: let $P_{(t-1)^{2}+1}: x_1x_2x_3...x_{(t-1)^{2}+1}$ be a path. Join
$(t-1)^{2}+1$ twin stems $x_1u_iu'_i$ where $1\leq i\leq
(t-1)^{2}+1$ each of length two with vertex $x_1$ of
$P_{(t-1)^{2}+1}$. This completes construction of $G$. We first show
that $Dist(G)=t$. For $t=2$, we have two twin stems attach with
$x_1$, and hence $Dist(G)=2$. For $t\ge 3$, we define a
labeling $f: V(G)\rightarrow \{1,2,3,...,t\}$ as follows:\\
$f(x_i)=t$, for all $i$, where $1\leq i \leq(t-1)^{2}+1.$
$$f({u_i})=\left\{
\begin{array}{ll}
1\,\,\,\,\,\,\,\,\,\,\,\,\,\,\,\,\,\,\,if \,\,\,\, 1\leq i\leq t-1, &  \\
2\,\,\,\,\,\,\,\,\,\,\,\,\,\,\,\,\,\,\,if \,\,\,\, t\leq i\leq 2(t-1),&  \\
3\,\,\,\,\,\,\,\,\,\,\,\,\,\,\,\,\,\,\,if \,\,\,\, 2t-1\leq i\leq
3(t-1),&  \\
\vdots \,\,\,\,\,\,\,\,\,\,\,\,\,\,\,\,\,\,\, \,\,\,\,\,\,\,\,\,\,\,\,\,\,\,\,\,\,\,\vdots \\
t-1\,\,\,\,\,\,\,\,\,\,if \,\,\,\, (t-1)(t-2)+1\leq i\leq
(t-1)^{2}, &  \\
t\,\,\,\,\,\,\,\,\,\,\,\,\,\,\,\,\,\,\,if \,\,\,\, i=(t-1)^{2}+1.
\end{array}
\right.$$

$$f({u'_i})=\left\{
\begin{array}{ll}
i\,\ \mbox{mod}(t-1)\,\,\,\,\,\,\,\,\,\,\,\,\,\,\,\,\,\,\,if \,\,\,\,1\leq i\,\ \mbox{mod}(t-1)\leq t-2\,\,\,\, and \,\,\,\, i\ne(t-1)^{2}+1,&  \\
t-1\,\,\,\,\,\,\,\,\,\,\,\,\,\,\,\,\,\,\,\,\,\,\,\,\,\,\,\,\,\,\,\,\,\,\,\,\,\,\,\,if \,\,\,\,i\,\ \mbox{mod}(t-1)=0,&  \\
t\,\,\,\,\,\,\,\,\,\,\,\,\,\,\,\,\,\,\,\hskip 2.2cm if \,\,\,\,
i=(t-1)^{2}+1.
\end{array}
\right.$$

Using this labeling, one can see the unique automorphism preserving
this labeling is the identity automorphism. Hence, $f$ is a
$t$-distinguishing. Since permutation with repetition of $t-1$
colors, when 2 of them are taken at a time is $2{{t-1}\choose
2}+(t-1)$, therefore $(t-1)^2+1$ twin stems can be labeled by at
least $t$-colors. Hence, $t$ is the least integer such that $G$ has
$t$-distinguishing labeling. Now, we denote the corresponding
vertices of $G_2$ as $v_{i}, v_{i}', y_i$ for all $i$, where $1\leq
i \leq(t-1)^{2}+1$ and construct a functigraph $F_{G}$ by defining
$g:V(G_1)\rightarrow V(G_2)$ as follows: $g(u_i)=g(u'_i)=y_i$, for
all $i$, where $1\leq i \leq (t-1)^{2}+1$ and $g(x_i)= g(y_i)$, for
all $i$, where $1\leq i \leq (t-1)^{2}+1$ as shown in the Figure
\ref{f1}. Thus, $F_G$ has only symmetries of $(t-1)^{2}+1$ twin
stems attach with $y_1$. Hence, $Dist(F_G)= t.$
\end{proof}

Consider an integer $t\geq 4$. We construct graph $G$ similarly as
in proof of Lemma \ref{fff} by taking a path
$P_{(t-3)^{2}+1}:x_1x_2...x_{(t-3)^{2}+1}$ and attach $(t-3)^{2}+1$
twin stems $x_1u_iu'_i$ where $1\leq i\leq (t-3)^{2}+1$ with any one
of its end vertex say $x_1$. Using similar labeling and arguments as
in proof of Lemma \ref{fff} one can see that $f$ is $t-2$
distinguishing and $t-2$ is least integer such that $G$ has $t-2$
distinguishing labeling. Define functigraph $F_{G}$, where $g:
V(G_1)\rightarrow V(G_2)$ is defined by: $g(u_i)= g(u'_i)=y_i$, for
all $i$, where $1\leq i \leq (t-3)^{2}+1$, $g(x_i)=v_i$, for all
$i$, where $1\leq i \leq (t-3)^{2}-1$, $g(x_i)=y_i$, for all $i$,
where $(t-3)^{2}\leq i \leq (t-3)^{2}+1$. From this construction,
$F_G$ has only symmetries of $2$ twin stems attach with $y_1$, and
hence $Dist(F_G)=2$. Thus, we have the following result which shows
that $Dist(G)+Dist(F_{G})$ can be arbitrary large:
\begin{Lemma}\label{fffw1}
For any integer $t\geq 4$, there exists a connected graph $G$ and a
function $g$ such that $Dist(G)+Dist(F_{G})=t$.
\end{Lemma}

Consider $t\geq 3$. We construct graph $G$ similarly as in proof of
Lemma \ref{fff} by taking a path $P_{4(t-1)^{2}+1}$:
$x_1x_2...x_{4(t-1)^{2}+1}$ and attach $4(t-1)^{2}+1$ twin stems
$x_1u_iu'_i$, where $1\leq i\leq 4(t-1)^{2}+1$ with $x_1$. Using
similar labeling and arguments as in proof of Lemma \ref{fff} one
can see that $f$ is $2t-1$ distinguishing and $2t-1$ is the least
integer such that $G$ has $2t-1$ distinguishing labeling. Let us now
define $g$ as $g(u_i)= g(u'_i)=y_i$, for all $i$, where $1\leq i\leq
4(t-1)^{2}+1$, $g(x_i)= v_i$, for all $i$, where $1\leq i\leq
3t^{2}-4t$ and $g(x_i)= y_i$, for all $i$, where $3t^{2}-4t+1\leq
i\leq 4(t-1)^{2}+1.$ Thus, $F_G$ has only symmetries of
$(t-2)^{2}+1$ twin stems attach with $y_1$, and hence
$Dist(F_G)=t-1$. After making this type of construction, we have the
following result which shows that $Dist(G)-Dist(F_{G})$ can be
arbitrary large:
\begin{Lemma}\label{fffw11}
For any integer $t\geq 3$, there exists a connected graph $G$ and a
function $g$ such that $Dist(G)-Dist(F_{G})=t$.
\end{Lemma}

\section{The distinguishing number of functigraphs of some families of graphs}

In this section, we discuss the distinguishing number of
functigraphs on complete graphs, edge deletion graphs of complete
graph and join of path graphs.

Let $G$ be the complete graph of order $n\ge3$ and $A$ and $B$ be
its two copies. We use following terminology for $F_G$ in proof of
Theorem \ref{f13}: Let $I=\{v_1,v_2,...,v_s\}$ and $n_i=|\{u\in A:
g(u)= v_i\}|$ for all $i$, where $1\le i \le s$. Also, let
$l=\mbox{max}\{n_i: 1\le i \le s\}$ and $m=|\{n_i: n_i=1, 1\le i \le
s \}|$. From the definitions of $l$ and $m$, we note that $2\le l
\le n-s+1$ and $0\le m \le s-1$.

Using function $\psi(m)$ as defined in previous section, we have
following lemma:
\begin{Lemma}\label{CompleteFunctiIdentity}
Let $G$ be the complete graph of order $n\ge 3$ and $g$ be a
bijective function, then $Dist(F_G)=\psi(n)$.
\end{Lemma}
\begin{proof}
Let $A=\{u_1,u_2,...,u_n\}$ and $I=\{g(u_1),g(u_2),...,g(u_n)\}=B$.
Also let $k=\psi(n)$. Let $f:V(F_G)\rightarrow \{1,2,...,k\}$ be a
labeling in which $f(u_i)$ is defined as in equation (\ref{array1})
and $f(g(u_i))$ as in equation (\ref{array2}) in proof of Lemma
\ref{LemmaTwinStem}. Using this labeling one can see that $f$ is a
$k$-distinguishing labeling for $F_G$. Since permutation with
repetition of $k$ colors, when 2 of them are taken at a time is
equal to $2{k\choose 2}+k$, therefore at least $k$ colors are needed
to label the vertices in $F_G$. Hence, $k$ is the least integer for
which $F_G$ has $k$-distinguishing labeling.
\end{proof}
Let $G$ be a complete graph and let $g:A\rightarrow B$ be a function
such that $2\le m\le s$. Without loss of generality assume
$u_1,u_2,...,u_m\in A$ are those vertices of $A$ such that
$g(u_i)\ne g (u_j)$ where $1\le i\ne j \le m$ in $B$. Also
$(u_iu_j)(g(u_i)g(u_j))\in \Gamma (F_G)$ for all $i\ne j$ where
$1\le i, j \le m$. By using similar labeling $f$ as defined in Lemma
\ref{CompleteFunctiIdentity}, at least $\psi(m)$ color are needed to
break these automorphism in $F_G$. Thus, we have following
proposition:
\begin{Proposition}\label{PropSi(m)}
Let $G$ be a complete graph of order $n\ge 3$ and $g$ be a function
such that $2\le m\le s$, then $Dist(F_G)\ge \psi(m)$.
\end{Proposition}

The following result gives the distinguishing number of functigraphs
of complete graphs.

\begin{Theorem}\label{f13} Let $G=K_{n}$ be the complete graph of order $n\geq
3$, and let $1<s\leq n-1$, then
$$Dist(F_{G})\in\{n-s,n-s+1,\psi(m)\}.$$
\end{Theorem}
\begin{proof}
We discuss following cases for $l$:
\begin{enumerate}
\item  If $l=n-s+1>2$, then $A$ contains $n-s+1$ twin vertices and
$B$ contains $n-s$ twin vertices (except for $n=3,4$ where $B$
contains no twin vertices). Also, there are $m(=s-1)$ vertices in
$A$ which have distinct images in $B$. These $m$ vertices and their
distinct images are labeled by at least $\psi(m)$ colors (only 1
color if $m=1$) by Proposition \ref{PropSi(m)}. Since $n-s+1$ is the
largest among $n-s+1$, $n-s$ and $\psi(m)$. Thus, $n-s+1$ is the
least number such that $F_G$ has $(n-s+1)$- distinguishing labeling.
Thus, $Dist(F_G)=n-s+1$.

\item If $l=n-s+1=2$, then $\psi(m)\ge \mbox{max}\{n-s+1, n-s\}$, and hence
$Dist(F_G)=\psi (m)$.

\item If $l<n-s$, then $B$
contains largest set of $n-s$ twin vertices in $F_G$. Also, there
are $m (\le s-2)$ vertices in $A$ each of which has distinct image
in $B$. Since $n-s\geq\psi(m)$, therefore $Dist(F_G)=n-s$.

\item If $l=n-s>2$, then both $A$ and $B$ contain
largest set of $n-s$ twin vertices in $F_G$. Also, there are
$m(=s-2)$ vertices in $A$ which have distinct images in $B$. Since
$n-s\geq\psi(m)$, therefore $Dist(F_G)=n-s$.

\item If $l=n-s=2$, then we take two subcases:
    \begin{enumerate}

        \item If $1<s\leq \lfloor \frac{n}{2}\rfloor +1$, then both
            $A$ and $B$ contain largest set of $n-s$ twin vertices in
            $F_G$. Also, there are $m(=s-2)$ vertices in $A$ which
            have distinct images in $B$. Since $n-s\geq\psi(m)$ (if $\psi(m)$ exists), therefore
            $Dist(F_G)=n-s$.
        \item If $\lfloor \frac{n}{2}\rfloor +1< s\leq n-1$, then $\psi(m)\ge \mbox{max}\{n-s+1, n-s\}$, and hence $Dist(F_G)=\psi(m)$.

    \end{enumerate}
\end{enumerate}
\end{proof}
Let $e^{\ast}$ be an edge of a connected graph $G$. Let
$G-ie^{\ast}$ is the graph obtained by deleting $i$ edges from graph
$G.$ A vertex $v$ of a graph $G$ is called \emph{saturated} if it is
adjacent to all other vertices of $G$.

 We define a function $\phi:
\mathbb{N} \rightarrow \mathbb{N}\setminus \{1\}$ as $\phi (i)=k$,
where $k$ is the least number such that $i\le {k\choose2}$. For
instance, $\phi(32)=9$. Note that $\phi$ is well defined.

\begin{Theorem}\label{f1122} Let $G$ be the complete graph of order $n\geq 5$ and $G_i=
G-ie^{\ast}$ for all $i$ where $1\leq i\leq\lfloor
\frac{n}{2}\rfloor$ and $e^{\ast}$ joins two saturated vertices of
the graph $G$. If $g$ is a constant function, then
$$Dist(F_{G_i})=max\{n-2i,\phi(i)\}.$$
\end{Theorem}
\begin{proof}
On deleting $i$ edges $e^*$ from $G$, we have $n-2i$ saturated
vertices and $i$ twin sets each of cardinality two. We will now show
that exactly $\phi(i)$ colors are required to label vertices of all
$i$ twin sets. We observe that, a vertex in a twin set can be mapped
on any one vertex in any other twin set. 
  Since two vertices
  in a twin set are labeled by a unique pair of colors out of $k\choose 2$ pairs of
  $k$ colors, therefore at least
 $k$ colors are required to label vertices of $i$ twin sets. Now, we discuss the following two cases for $\phi(i)$:
 \begin{enumerate}
 \item If $\phi(i)\le n-2i$, then number of colors required to label $n-2i$
  saturated vertices is greater than or equal to number of colors required to
  label vertices of $i$ twin sets. Thus, we label $n-2i$ saturated vertices
  with exactly $n-2i$ colors and out of these $n-2i$ colors, $\phi(i)$ colors will be used
  to label vertices of $i$ twin sets.
 \item If $\phi(i)> n-2i$, then number of colors required to label $n-2i$
  saturated vertices is less than the number of colors required to
  label vertices of $i$ twin sets. Thus, we label vertices
  of $i$ twin sets with $\phi(i)$ colors and out of these $\phi(i)$ colors, $n-2i$ colors will be used
  to label saturated vertices in $G_i$.
 \end{enumerate}

If $g$ is constant, then by using same arguments as in the proof of
Lemma \ref{Lemma g constant}, $Dist(F_{G_i})=Dist(G_i).$
\end{proof}


Suppose that $G=(V_{1},E_{1})$ and $G^*=(V_{2},E_{2})$ be two graphs
with disjoint vertex sets $V_{1}$ and $V_{2}$ and disjoint edge sets
$E_{1}$ and $E_{2}$. The \emph{join} of $G$ and $G^*$ is the graph
$G+G^*$, in which $V(G+G^*)=V_{1}\cup V_{2}$ and $E(G+G^*)=E_{1}\cup
E_{2}\cup \{$ $uv$: $u\in V_1$, $v\in V_2\}$.

%
%
%

\begin{Theorem} \cite{Alikhani} Let $G$ and $G^*$ be two connected
graphs, then $Dist(G+G^*)\geq max\{Dist(G), Dist(G^*)\}.$
\end{Theorem}


\begin{Proposition}
Let $P_n$ be a path graph of order $n\ge 2$, then for all $m,n\ge 2$
and $1<s<m+n$, $1\le Dist(F_{P_m+P_n})\le 3$.
\end{Proposition}
\begin{proof}
Let $P_m: v_1,...,v_m$ and $P_n:u_1,...,u_n$. We discuss following
cases for $m,n$.
\begin{enumerate}
\item If $m=2$ and $n=2$, then $P_2+P_2=K_4$, and hence $1\leq Dist(F_{K_4})\le 3$
 by Theorem \ref{f13}.
\item If $m=2$ and $n=3$, then $P_2+P_3$ has 3 saturated
 vertices. Thus, $1\le Dist(F_{P_2+P_3})\le 4$ by Proposition
 \ref{f5}. However, for all $s$ where $2 \le s \le 4$ and all possible definitions of $g$
 in $F_{P_2+P_3}$, one can see $1 \le Dist(F_{P_2+P_3}) \le 3$.

\item If $m=3$ and $n=3$, then a labeling $f:V(P_3+P_3)\rightarrow \{1,2,3\}$ defined as:

$$f(x)=\left\{
\begin{array}{ll}
1\,\,\,\,\,\,\,\,\,\,\,\,\,\,if \,\,\,\,\,\,\,\,\,\,\,\,\, x=v_1,v_2\\
2\,\,\,\,\,\,\,\,\,\,\,\,\,\,if\,\,\,\,\,\,\,\,\,\,\,\,\, x= v_3,u_3\\
3\,\,\,\,\,\,\,\,\,\,\,\,\,\,if\,\,\,\,\,\,\,\,\,\,\,\,\, x= u_1,u_2\\
\end{array} \right.$$

is a distinguishing labeling for $P_3+P_3$, and hence
$Dist(P_3+P_3)= 3$. Thus, $1\le Dist(F_{P_3+P_3})\le 4$ by
Proposition \ref{f5}. However, for all $s$ where $2 \le s \le 5$ and
all possible definitions of $g$ in $F_{P_3+P_3}$, one can see $1 \le
Dist(F_{P_3+P_3}) \le 3$.

\item If $m\ge 2$ and $n\ge 4$, then a labeling $f:V(P_m+P_n)\rightarrow \{1,2\}$ defined as:

$$f(x)=\left\{
\begin{array}{ll}
1\,\,\,\,\,\,\,\,\,\,\,\,\,\,if \,\,\,\,\,\,\,\,\,\,\,\,\, x=v_1,u_2,...,u_n\\
2\,\,\,\,\,\,\,\,\,\,\,\,\,\,if\,\,\,\,\,\,\,\,\,\,\,\,\, x=u_1,
v_2,...,v_m &
\end{array} \right.$$

is a distinguishing labeling for $P_m+P_n$, and hence
$Dist(P_m+P_n)= 2$. Thus, result follows by Proposition \ref{f5}.

\end{enumerate}


\end{proof}


\end{document}